\documentclass[12pt]{amsart}
\usepackage{amscd,amsmath,amsthm,amssymb}
\usepackage{color}
\usepackage{pstricks}
\usepackage{stmaryrd}
\usepackage{tikz}
\usepackage{url}

\usepackage{latexsym}
\usepackage{amsfonts,amsmath,mathtools}
\usepackage{graphics}
\usepackage{float}
\usepackage{enumitem}

\newpsstyle{fatline}{linewidth=1.5pt}
\newpsstyle{fyp}{fillstyle=solid,fillcolor=verylight}
\definecolor{verylight}{gray}{0.97}
\definecolor{light}{gray}{0.9}
\definecolor{medium}{gray}{0.85}
\definecolor{dark}{gray}{0.6}

\def\NZQ{\mathbb}               

\def\ZZ{{\NZQ Z}}

\def\G{{\mathcal G}}

\def\HS{\textup{HS}}
\def\pd{\textup{proj}\phantom{.}\!\textup{dim}}

\def\opn#1#2{\def#1{\operatorname{#2}}} 
\opn\chara{char} \opn\length{\ell} \opn\pd{pd} \opn\rk{rk}
\opn\projdim{proj\,dim} \opn\injdim{inj\,dim} \opn\rank{rank}
\opn\depth{depth} \opn\grade{grade} \opn\height{height}
\opn\embdim{emb\,dim} \opn\codim{codim}

\opn\Tr{Tr} \opn\bigrank{big\,rank}
\opn\superheight{superheight}\opn\lcm{lcm}
\opn\trdeg{tr\,deg}
\opn\reg{reg} \opn\lreg{lreg} \opn\ini{in} \opn\lpd{lpd}
\opn\size{size} \opn\sdepth{sdepth}
\opn\link{link}\opn\fdepth{fdepth}\opn\lex{lex}
\opn\tr{tr}
\opn\type{type}
\opn\gap{gap}
\opn\diam{diam}
\opn\Mod{Mod}
\opn\div{div} \opn\Div{Div} \opn\cl{cl} \opn\Cl{Cl}
\opn\Spec{Spec} \opn\Supp{Supp} \opn\supp{supp} \opn\Sing{Sing}
\opn\Ass{Ass} \opn\Min{Min}\opn\Mon{Mon}
\opn\Ann{Ann} \opn\Rad{Rad} \opn\Soc{Soc}
\opn\Im{Im} \opn\Ker{Ker} \opn\Coker{Coker} \opn\Am{Am}
\opn\Hom{Hom} \opn\Tor{Tor} \opn\Ext{Ext} \opn\End{End}
\opn\Aut{Aut} \opn\id{id}

\opn\nat{nat}
\opn\pff{pf}
\opn\Pf{Pf} \opn\GL{GL} \opn\SL{SL} \opn\mod{mod} \opn\ord{ord}
\opn\Gin{Gin} \opn\Hilb{Hilb}\opn\sort{sort}
\opn\PF{PF}\opn\Ap{Ap}
\opn\dist{dist}
\opn\aff{aff}
\opn\relint{relint} \opn\st{st}
\opn\lk{lk} \opn\cn{cn} \opn\core{core} \opn\vol{vol}  \opn\inp{inp} \opn\nilpot{nilpot}
\opn\link{link} \opn\star{star}\opn\lex{lex}\opn\set{set}
\opn\width{wd}
\opn\Fr{F}
\opn\QF{QF}
\opn\G{G}
\opn\type{type}\opn\res{res}
\opn\conv{conv}
\opn\sr{sr}
\opn\gr{gr}

\def\pot#1#2{#1[\kern-0.28ex[#2]\kern-0.28ex]}

\opn\dirlim{\underrightarrow{\lim}}
\opn\inivlim{\underleftarrow{\lim}}
\def\Implies{\ifmmode\Longrightarrow \else
	\unskip${}\Longrightarrow{}$\ignorespaces\fi}
\def\implies{\ifmmode\Rightarrow \else
	\unskip${}\Rightarrow{}$\ignorespaces\fi}
\def\iff{\ifmmode\Longleftrightarrow \else
	\unskip${}\Longleftrightarrow{}$\ignorespaces\fi}

\let\:=\colon
\newtheorem{Theorem}{Theorem}[section]

\newtheorem{Corollary}[Theorem]{Corollary}
\newtheorem{Proposition}[Theorem]{Proposition}

\newtheorem{Example}[Theorem]{Example}
\newtheorem{Examples}[Theorem]{Examples}

\newtheorem{Conjecture}[Theorem]{Conjecture}
\newtheorem{Question}[Theorem]{Question}

\let\epsilon\varepsilon
\let\kappa=\varkappa
\def\qed{\ifhmode\textqed\fi
	\ifmmode\ifinner\hfill\quad\qedsymbol\else\dispqed\fi\fi}
\def\textqed{\unskip\nobreak\penalty50
	\hskip2em\hbox{}\nobreak\hfill\qedsymbol
	\parfillskip=0pt \finalhyphendemerits=0}
\def\dispqed{\rlap{\qquad\qedsymbol}}

\opn\dis{dis}
\def\pnt{{\raise0.5mm\hbox{\large\bf.}}}

\opn\Lex{Lex}
\opn\Shad{Shad}

\usepackage{lipsum}

\begin{document}

	\title{Shellability of Componentwise Discrete Polymatroids}
	\author{Antonino Ficarra}
	
	\address{Antonino Ficarra, Department of mathematics and computer sciences, physics and earth sciences, University of Messina, Viale Ferdinando Stagno d'Alcontres 31, 98166 Messina, Italy}
	\email{antficarra@unime.it}
	
	\thanks{.
	}
	
	
	
	\maketitle
	
	\begin{abstract}
		In the present paper, motivated by a conjecture of Jahan and Zheng, we prove that componentwise polymatroidal ideals have linear quotients. This solves positively a conjecture of Bandari and Herzog.
	\end{abstract}

	\section{Componentwise linear quotients}
	
	Let $S=K[x_1,\dots,x_n]$ be the polynomial ring with coefficients over a field $K$, and let $I\subset S$ be a monomial ideal. Let $\mathcal{G}(I)$ be the unique minimal set of monomial generators of $I$. We say that $I$ has \textit{linear quotients} if there exists an order $u_1,\dots,u_m$ of $\mathcal{G}(I)$ such that $(u_1,\dots,u_{j-1}):u_j$ is generated by variables for $j=2,\dots,m$.\smallskip
	
	For $j\ge0$, let $I_{\langle j\rangle}$ be the monomial ideal generated by the monomials of degree $j$ belonging to $I$. We say that $I$ has \textit{componentwise linear quotients} if $I_{\langle j\rangle}$ has linear quotients for all $j$. It is known that ideals with linear quotients have componentwise linear quotients \cite[Corollary 2.8]{JZ10}. The converse is an open question \cite{JZ10}:
	\begin{Conjecture}\label{Conj:JZ}
		\textup{(Jahan--Zheng)} Let $I$ be a monomial with componentwise linear quotients. Then $I$ has linear quotients.
	\end{Conjecture}
	
	The above conjecture is widely open. See \cite{Shen} for some partial results.
	
	\section{Componentwise Polymatroidal Ideals}
	
	A monomial ideal $I$ is called \textit{polymatroidal} if the set of the exponent vectors of the minimal monomial generators of $I$ is the set of bases of a discrete polymatroid \cite{JT}. Polymatroidal ideals have linear quotients. A monomial ideal $I$ is \textit{componentwise polymatroidal} if the component $I_{\langle j\rangle}$ is polymatroidal for all $j$. Hence, componentwise polymatroidal ideals are ideals with componentwise linear quotients. Therefore, a particular case of Conjecture \ref{Conj:JZ} is:
	\begin{Conjecture}\label{Con:BH}
		\textup{(Bandari--Herzog)} Let $I$ be a componentwise polymatroidal ideal. Then $I$ has linear quotients.
	\end{Conjecture}
	
	This conjecture was firstly considered in \cite{BH2013} and proved for ideals of componentwise Veronese type. Recently, Bandari and Qureshi \cite{BQ23} proved it in the two variables case and for componentwise polymatroidal ideals with strong exchange property.\medskip
	
	We are going to prove Conjecture \ref{Con:BH} in full generality.\smallskip
	
	For this aim, we recall some results from \cite{BQ23}. For a monomial $u=x_1^{a_1}\cdots x_n^{a_n}\in S$, we denote its \textit{degree} by $\deg(u)=a_1+\dots+a_n$. Whereas, the \textit{$x_i$-degree} of $u$ is the integer $\deg_{x_i}(u)=a_i=\max\{j\ge0:x_i^j\ \text{divides}\ u\}$.
	\begin{Theorem}\label{Thm:BQ}
		\textup{\cite[Proposition 1.2]{BQ23}} Let $I\subset S$ be a monomial ideal. Then, the following conditions are equivalent.
		\begin{enumerate}
			\item[\textup{(i)}] $I$ is a componentwise polymatroidal ideal.
			\item[\textup{(ii)}] For all $u,v\in I$ with $\deg(u)\le\deg(v)$ and with $u$ not diving $v$, and all $i$ such that $\deg_{x_i}(v)>\deg_{x_i}(u)$ there exists an integer $j$ with $\deg_{x_j}(v)<\deg_{x_j}(u)$ and such that $x_j(v/x_i)\in I$.
		\end{enumerate}
	\end{Theorem}
	
	\begin{Proposition}\label{Prop:BQ}
		\textup{\cite[Proposition 1.5]{BQ23}} Let $I\subset S$ be a componentwise polymatroidal ideal. Then the following property, called the dual exchange property, holds: For all $u,v\in I$ with $\deg(u)\le\deg(v)$, and all $i$ such that $\deg_{x_i}(v)<\deg_{x_i}(u)$ there exists an integer $j$ with $\deg_{x_j}(v)>\deg_{x_j}(u)$ and such that $x_i(v/x_j)\in I$.
	\end{Proposition}
	
	We close this section with some examples.
	\begin{Examples}\label{Ex:CP}
		\rm (a) Componentwise polymatroidal ideals in two variables were classified in \cite{BQ23}. Let $I\subset K[x,y]$ be a monomial ideal. We may assume that the minimal monomial generators of $I$ do not have any common factor. In fact, if $I=uJ$ for a monomial $u\in S$ and a monomial ideal $J$, then $I$ is componentwise polymatroidal if and only if $J$ is such. It is proved in \cite[Corollary 2.7]{BQ23} that $I\subset K[x,y]$ is a componentwise polymatroidal ideal if and only if $I$ is a \textit{$yx$-tight} ideal in the sense of \cite[Definition 2.1]{BQ23}.
		
		(b) Let ${\bf a}=(a_1,\dots,a_n)\in\ZZ_{\ge0}^n$ and $d\ge1$. The ideal of \textit{Veronese type} $({\bf a},d)$ is
		$$
		I_{{\bf a},d}\ =\ (x_1^{b_1}\cdots x_n^{b_n}\ :\ b_1+\dots+b_n=d,\ b_i\le a_i,\ \textup{for all}\ i).
		$$
		Monomial ideals whose all components are of Veronese type are componentwise polymatroidal ideals, see also \cite[Section 3]{BH2013}.\smallskip
		
		(c) A monomial ideal $I$ generated in a single degree has the strong exchange property if for all $u,v\in\mathcal{G}(I)$ all $i$ such that $\deg_{x_i}(u)>\deg_{x_i}(v)$ and all $j$ such that $\deg_{x_j}(u)<\deg_{x_j}(v)$, then $x_j(u/x_i)$ belongs to $\mathcal{G}(I)$. It is known that any such ideal $I$ is a polymatroidal ideal of the form $I=uI_{{\bf a},d}$ for some suitable monomial $u\in S$, ${\bf a}\in\ZZ_{\ge0}^n$ and $d\ge1$. Hence, ideals whose all components satisfy the strong exchange property are componentwise polymatroidal.\smallskip
		
		(d) Denote by $\mathfrak{m}$ the maximal ideal $(x_1,\dots,x_n)$. It is known that the product of polymatroidal ideals is polymatroidal. Let $1\le d_1<\dots<d_t$ be positive integers, $J_1,\dots,J_t$ be polymatroidal ideals generated in degrees $d_1,\dots,d_t$, respectively, such that $\mathfrak{m}^{d_{i+1}-d_i}J_i\subseteq J_{i+1}$ for $i=1,\dots,t-1$. Let $I=J_1+\dots+J_t$. Then $I$ is componentwise polymatroidal. Indeed,
		$$
		I_{\langle j\rangle}=\begin{cases}
			J_{i}&\textup{if}\ j=d_i,\ \textup{for some}\ i,\\
			\mathfrak{m}^{j-d_i}J_{i}&\textup{if}\ d_i<j<d_{i+1},\ \textup{for some}\ i,\\
			\mathfrak{m}^{j-d_t}J_{t}&\textup{if}\ j\ge d_t,
		\end{cases}
		$$
		is polymatroidal for all $j$.\smallskip
		
		(e) Let $u=x_{i_1}\cdots x_{i_d}$ and $v=x_{j_1}\cdots x_{j_d}$ be two monomials of the same degree $d$, with $1\le i_1\le\dots\le i_d\le n$ and $1\le j_1\le\dots\le j_d\le n$. We write $v\preceq_{\textup{Borel}}u$ if $j_k\le i_k$ for all $k$. The \textit{principal Borel ideal generated by $u$}, denoted by $B(u)$, is the monomial ideal generated in degree $d$ whose minimal generating set is
		$$
		\mathcal{G}(B(u))\ =\ \{v\in S\ :\ \deg(v)=\deg(u),\ v\preceq_{\textup{Borel}}u\}.
		$$ 
		It is known that $B(u)$ is polymatroidal. Let $u,v\in S$ be monomials of the same degree. It follows from the definition of $\succeq_{\textup{Borel}}$ that $B(v)\subseteq B(u)$ if and only if $v\preceq_{\textup{Borel}}u$. Notice that $\mathfrak{m}^\ell B(u)=B(ux_n^\ell)$ for any $\ell$. We say that a monomial ideal $I$ is \textit{componentwise principal Borel} if all $I_{\langle j\rangle}$ are principal Borel ideals. From (d) and these considerations, it follows that $I$ is componentwise principal Borel if and only if there exists monomials $u_1,\dots,u_t$ of degrees $d_1<\dots<d_t$, respectively, such that
		$$
		u_ix_n^{d_{i+1}-d_i}\ \preceq_{\textup{Borel}}\ u_{i+1},
		$$
		for $i=1,\dots,t-1$. In particular, in such a case $I=B(u_1)+\dots+B(u_t)$.\smallskip
		
		(f) Actually, componentwise polymatroidal ideals appeared implicitly for the first time in the work of Francisco and Van Tuyl \cite{FVT2007}, in connection to \textit{ideals of fat points}. For $n\ge1$, set $[n]=\{1,\dots,n\}$. Given a non-empty subset $A$ of $[n]$, denote by $P_A$ the polymatroidal ideal $(x_i:i\in A)$. Suppose that $A_1,\dots,A_t$ are non-empty subsets of $[n]$ such that $A_i\cup A_j=[n]$ for all $i\ne j$. It is shown in \cite[Theorem 3.1]{FVT2007} that $$I=P_{A_1}^{k_1}\cap\dots\cap P_{A_t}^{k_t}$$ is componentwise polymatroidal for all positive integers $k_1,\dots,k_t\ge1$.\smallskip
		
		(g) Let $I$ be a polymatroidal ideal generated in degree $d$. The \textit{socle} of $I$ is the monomial ideal $\textup{soc}(I)=(I:\mathfrak{m})_{\langle d-1\rangle}$. It is conjectured in \cite[page 760]{BH2013}, and proved in some special cases in \cite{F2}, that $\textup{soc}(I)$ is again polymatroidal. It is noted in \cite{F2} that $(I:\mathfrak{m})$ is generated in at most two degrees $d-1$ and $d$, and that $(I:\mathfrak{m})_{\langle d\rangle}=I$. Thus
		$$
		(I:\mathfrak{m})=\textup{soc}(I)+I.
		$$
		Furthermore, it follows by the very definition of colon ideal that $\mathfrak{m}(I:\mathfrak{m})\subseteq I$. In particular, $\mathfrak{m}\cdot\textup{soc}(I)\subseteq I$. Hence, if $\textup{soc}(I)$ is polymatroidal, it would follow by the construction in (d) that $(I:\mathfrak{m})$ is componentwise polymatroidal.\smallskip
		
		(h) More generally, let $I$ be a componentwise polymatroidal ideal. If the above conjecture about the socle of polymatroidal ideals is true, then $(I:\mathfrak{m})$ would be componentwise polymatroidal as well. Indeed,
		\begin{align*}
			(I:\mathfrak{m})_{\langle j\rangle}\ &=\ \{u\in S\ :\ \deg(u)=j,\ \textup{and}\ ux_i\in I,\ \textup{for all}\ i\}\\
			&=\ \{u\in S\ :\ \deg(u)=j,\ \textup{and}\ ux_i\in I_{\langle j+1\rangle},\ \textup{for all}\ i\}\\
			&=\ (I_{\langle j+1\rangle}:\mathfrak{m})_{\langle j\rangle}\\
			&=\ \textup{soc}(I_{\langle j+1\rangle})
		\end{align*}
	    would be a polymatroidal ideal, for all $j$.
	\end{Examples}

	\section{Componentwise polymatroidal ideals have linear quotients}
	
	We are now ready to prove the main result in the paper.
	
	\begin{Theorem}\label{Thm:CMLQ}
		Componentwise polymatroidal ideals have linear quotients.
	\end{Theorem}
	\begin{proof}
		Let $I\subset S=K[x_1,\dots,x_n]$ be a componentwise polymatroidal ideal. We prove the theorem by induction on $n$, the number of variables.
		
		For $n=1$, $I$ is a principal ideal and it has linear quotients.
		
		Let $n>1$. If $|\mathcal{G}(I)|=1$, $I$ is again a principal ideal. Suppose $|\mathcal{G}(I)|>1$. By induction, all componentwise polymatroidal ideals in $S$ with less than $|\mathcal{G}(I)|$ generators have linear quotients. Furthermore, we may suppose that all monomials $u\in \mathcal{G}(I)$ have no common factor $w\ne1$. Otherwise, we may consider the ideal $I'$ with $\mathcal{G}(I')=\{u/w:u\in \mathcal{G}(I)\}$. Then $I'$ is componentwise polymatroidal too, and $I$ has linear quotients if and only if $I'$ has linear quotients. Let $d=\alpha(I)$ be the \textit{initial degree} of $I$. That is, $I_{\langle j\rangle}=0$ for $0\le j<d$ and $I_{\langle d\rangle}\ne0$. Let $j$ any integer such that $x_j$ divides some monomial generator of $I_{\langle d\rangle}$. After a suitable relabeling, we may assume $j=1$. Therefore, we can write
		$$
		I=x_1I_1+I_2
		$$
		for unique monomial ideals $I_1,I_2\subset S$ such that
		\begin{align*}
		\mathcal{G}(x_1I_1)\ &=\ \{u\in \mathcal{G}(I)\ :\ x_1\ \textup{divides}\ u\},\\
		\mathcal{G}(I_2)\ &=\ \{u\in \mathcal{G}(I)\ :\ x_1\ \textup{does not divide}\ u\}.
		\end{align*}
		
		We are going to prove the following three facts:
		\begin{enumerate}
		\item[(a)] $I_2\subseteq I_1$ as monomial ideals of $S$.
		\item[(b)] $x_1I_1$ is a componentwise polymatroidal ideal of $S$.
		\item[(c)] $I_2$ is a componentwise polymatroidal ideal of $K[x_2,\dots,x_n]$.
		\end{enumerate}
		
		Once we get these claims, the proof ends as follows. Since the monomials in $\mathcal{G}(I)$ have no common factor $\ne1$, $|\mathcal{G}(x_1I_1)|$ and $|\mathcal{G}(I_2)|$ are strictly less than $|\mathcal{G}(I)|$. Items (b) and (c) together with our induction hypothesis imply that $x_1I_1$ and $I_2$ have linear quotients, with linear quotients orders, say $u_1,\dots,u_r$ of $\mathcal{G}(x_1I_1)$, and $v_1,\dots,v_s$ of $\mathcal{G}(I_2)$. We claim $u_1,\dots,u_r,v_1,\dots,v_s$ is a linear quotients order of $I$. Indeed, if $\ell\in[r]$, then $(u_1,\dots,u_{\ell-1}):u_\ell$ is generated by variables by our inductive hypothesis on $x_1I_1$. Whereas, if $\ell\in[s]$, using the inductive hypothesis on $I_2$, we obtain that the ideal
		\begin{align*}
		(u_1,\dots,u_r,v_1,\dots,v_{\ell-1}):v_{\ell}&=(u_1,\dots,u_r):v_\ell+(v_1,\dots,v_{\ell-1}):v_{\ell}\\
		&=(x_1I_1:v_\ell)+(v_1,\dots,v_{\ell-1}):v_\ell\\
		&=(x_1)+(v_1,\dots,v_{\ell-1}):v_\ell
		\end{align*}
		is generated by variables, because it is a sum of ideals generated by variables. Here, we have used the fact that $v_\ell\in \mathcal{G}(I_2)\subset I_1$ and $x_1$ does not divide $v_\ell$ to get the equality $(x_1I_1:v_\ell)=x_1(I_1:v_\ell)=x_1S=(x_1)$.
		\medskip
		
		It remains to prove items (a), (b) and (c).
		\medskip
		
		\textit{Proof of} (a): It is enough to show that any monomial of $\mathcal{G}(I_2)$ is divided by some monomial of $I_1$. Let $v\in \mathcal{G}(I_2)$ and let $u\in x_1I_1$ with $\deg(u)=\alpha(I)$. Then $\deg(u)=\alpha(x_1I_1)=\alpha(I)$. Therefore $\deg(u)\le\deg(v)$. Moreover $\deg_{x_1}(v)=0<\deg_{x_1}(u)$. By the dual exchange property (Proposition \ref{Prop:BQ}) we can find $j$ with $\deg_{x_j}(v)>\deg_{x_j}(u)$ such that $x_1(v/x_j)\in I$. Then there is $w\in \mathcal{G}(I)$ that divides $x_1(v/x_j)$. If $w\in \mathcal{G}(I_2)$, then $x_1$ does not divide $w$ and so $w$ divides $v/x_j$, against the fact that $v$ is a minimal generator of $I$. Hence $w\in \mathcal{G}(x_1I_1)$ and $w=x_1w'$ divides $x_1(v/x_j)$. Consequently $w'\in I_1$ divides $v/x_j$. Hence $w'\in I_1$ divides $v\in \mathcal{G}(I_2)$, as desired. 
		\medskip
		
		\textit{Proof of} (b): Let $u,v\in x_1I_1$ with $\deg(u)\le\deg(v)$, $u$ not diving $v$, and let $i$ such that $\deg_{x_i}(v)>\deg_{x_i}(u)$. By Theorem \ref{Thm:BQ}(ii) it is enough to determine $j$ with $\deg_{x_j}(v)<\deg_{x_j}(u)$ such that $x_j(v/x_i)\in x_1I_1$. Since $u,v\in I$, by Theorem \ref{Thm:BQ} we can find $j$ with $\deg_{x_j}(v)<\deg_{x_j}(u)$ such that $x_j(v/x_i)\in I$. We show now that $x_j(v/x_i)\in x_1I_1$. Note that $x_1$ divides $v\in x_1I_1$. If $i\ne1$, then $x_1$ divides $x_j(v/x_i)$. Otherwise, if $i=1$, since $x_1$ divides $u\in x_1I_1$ and $\deg_{x_1}(v)>\deg_{x_1}(u)\ge1$, we obtain $\deg_{x_1}(x_j(v/x_1))\ge1$. Hence, in both cases $x_1$ divides $x_j(v/x_i)$. Now, if some $w\in \mathcal{G}(I_2)$ divides $x_j(v/x_i)$ then $x_1w$ also divides $x_j(v/x_i)$. By item (a), $x_1w\in x_1I_2\subset x_1I_1$ and so $x_j(v/x_i)\in x_1I_1$. Otherwise, some $w\in \mathcal{G}(x_1I_1)$ divides $x_j(v/x_i)$ and again $x_j(v/x_i)\in x_1I_1$, as wanted.
		\medskip
		
		\textit{Proof of} (c): Let $u,v\in I_2$ with $\deg(u)\le\deg(v)$, $u$ not diving $v$ and let $i$ such that $\deg_{x_i}(v)>\deg_{x_i}(u)$. Recall that we are regarding $I_2$ as an ideal of $K[x_2,\dots,x_n]$, hence $\deg_{x_1}(v)=\deg_{x_1}(u)=0$. By Theorem \ref{Thm:BQ}(ii) valid in $I$, there exists $j$ with $\deg_{x_j}(v)<\deg_{x_j}(u)$ and such that $x_j(v/x_i)\in I$. Since $j\ne 1$, $x_1$ does not divide $x_j(v/x_i)$. Hence $x_j(v/x_i)\in I_2$, as desired.
	\end{proof}
	\begin{Example}
		\rm By Examples \ref{Ex:CP}(f), $I=P_{\{1,2,3\}}^2\cap P_{\{1,3,4\}}^2$ is componentwise polymatroidal. Notice that $\mathcal{G}(I)=\{x_1^2,x_1x_3,x_3^2,x_1x_2x_4,x_2x_3x_4,x_2^2x_4^2\}$ and $\alpha(I)=2$. A variable dividing a generator of least degree is for instance $x_1$. Using the notation in the proof of Theorem \ref{Thm:CMLQ} and the \textit{Macaulay2} \cite{GDS} package \cite{FPack1}, we checked that $I_1=(x_1,x_3,x_2x_4)$, $I_2=(x_3^2,x_2x_3x_4,x_2^2x_4^2)$ are componentwise polymatroidal ideals and $I_2\subseteq I_1$. The ideal $I_1$ has linear quotients order $x_1,x_3,x_2x_4$. Whereas a linear quotients order of $I_2$ is $x_3^2,x_2x_3x_4,x_2^2x_4^2$. Hence, according to the proof of the theorem, a linear quotients order of $I=x_1I_1+I_2$ is indeed $x_1^2,x_1x_3,x_1x_2x_4,x_3^2,x_2x_3x_4,x_2^2x_4^2$.
	\end{Example}
	
	Unfortunately the product of componentwise polymatroidal ideals is not a componentwise polymatroidal ideal anymore \cite{BH2013}. However, we expect that
	\begin{Conjecture}
		Each power of a componentwise polymatroidal ideal has linear quotients.
	\end{Conjecture}

    For a monomial ideal $I$, denote by $\HS_j(I)$ the $j$th \textit{homological shift ideal} of $I$ \cite{F2}. That is, the monomial ideal generated by the monomials whose exponent vectors are the $j$th multigraded shifts appearing in the minimal multigraded free resolution of $I$. It is expected that $\HS_j(I)$ is polymatroidal for all $j$, if $I$ is polymatroidal. For some partial results on this conjecture see \cite{Bay019,F2,FH2023}.
    \begin{Question}
    	Let $I$ be a componentwise polymatroidal ideal. Is $\HS_j(I)$ componentwise polymatroidal as well, for all $j$?
    \end{Question}
	
	\section{Componentwise Discrete Polymatroids}
	
	In this final section, we introduce the combinatorial counterpart of componentwise polymatroidal ideals, which we call \textit{componentwise discrete polymatroids}.\medskip
	
	For ${\bf a}=(a_1,\dots,a_n)\in\ZZ_{\ge0}^n$, denote by ${\bf a}[i]=a_i$ the $i$th component of ${\bf a}$. We set $|{\bf a}|=a_1+\dots+a_n$. Let ${\bf a},{\bf b}\in\ZZ_{\ge0}^n$. We write ${\bf a}\le{\bf b}$ if ${\bf a}[i]\le{\bf b}[i]$ for all $i$. We write ${\bf a}<{\bf b}$ if ${\bf a}\le{\bf b}$ and ${\bf a}\ne{\bf b}$. Let ${\bf e}_1,\dots,{\bf e}_n$ be the canonical basis of $\ZZ_{\ge0}^n$, that is ${\bf e}_i[j]=0$ for all $j\ne i$ and ${\bf e}_i[i]=1$. A \textit{simplicial multicomplex} $\mathcal{M}$ on the vertex set $[n]$ is a finite subset of $\ZZ_{\ge0}^n$ satisfying the following properties:
	\begin{enumerate}
		\item[(a)] If ${\bf a}\in\mathcal{M}$ and ${\bf b}\le{\bf a}$, then ${\bf b}\in\mathcal{M}$.
		\item[(b)] ${\bf e}_i\in\mathcal{M}$ for all $i$.
	\end{enumerate}
    Any ${\bf a}\in\mathcal{M}$ is called a \textit{face} of $\mathcal{M}$. A \textit{facet} ${\bf a}\in\mathcal{M}$ is a face of $\mathcal{M}$ for which there is no ${\bf b}\in\mathcal{M}$ such that ${\bf a}<{\bf b}$. The set of facets of $\mathcal{M}$ is denoted by $\mathcal{F}(\mathcal{M})$. We set $\alpha(\mathcal{M})=\min\{|{\bf a}|:{\bf a}\in\mathcal{F}(\mathcal{M})\}$ and $\omega(\mathcal{M})=\max\{|{\bf a}|:{\bf a}\in\mathcal{F}(\mathcal{M})\}$. The dimension of $\mathcal{M}$ is $\dim(\mathcal{M})=\max\{|{\bf a}|-1:{\bf a}\in\mathcal{M}\}$. Notice that $\dim(\mathcal{M})=\omega(\mathcal{M})-1$.
    
    For any ${\bf b}_1,\dots,{\bf b}_\ell\in\ZZ_{\ge0}^n$, we denote by $\langle{\bf b}_1,\dots,{\bf b}_\ell\rangle$ the unique, smallest with respect to the inclusion, simplicial multicomplex containing ${\bf b}_1,\dots,{\bf b}_\ell$.

    For ${\bf a}\in\ZZ_{\ge0}$, we set ${\bf x^a}=\prod_ix_i^{{\bf a}[i]}$. The \textit{facet ideal} of $\mathcal{M}$ is defined as
    $$
    I(\mathcal{M})\ =\ ({\bf x^a}\ :\ {\bf a}\in\mathcal{F}(\mathcal{M})).
    $$
    
    There is a natural bijection between monomial ideals of $S$ and simplicial multicomplexes on vertex set $[n]$, defined by assigning to each monomial ideal $I\subset S$ the simplicial multicomplex $\mathcal{M}_I=\langle{\bf a}\in\ZZ_{\ge0}^n:{\bf x^a}\in\mathcal{G}(I)\rangle$.\smallskip
    
    Now, we introduce a special class of simplicial multicomplexes. A simplicial multicomplex $\mathcal{P}$ is called a \textit{componentwise discrete polymatroid} if $I(\mathcal{P})$ is a componentwise polymatroidal ideal. To adhere to the classical terminology used for discrete polymatroids, we call the facets of $\mathcal{P}$ the \textit{bases} of $\mathcal{P}$. Notice that a componentwise discrete polymatroid is a discrete polymatroid if and only if $\alpha(\mathcal{P})=\omega(\mathcal{P})$.
    
    We denote by $[n]^{\langle d\rangle}$ the discrete polymatroid $\{{\bf a}\in\ZZ_{\ge0}^n:|{\bf a}|\le d\}$. In particular $[n]^{\langle1\rangle}=\{{\bf e}_1,\dots,{\bf e}_n\}$. Whereas, given a non-empty finite set $A\subset\ZZ_{\ge0}^n$ and an integer $j\ge0$, we set $A_{\langle j\rangle}=\{{\bf a}\in A:|{\bf a}|\le j\}$. Furthermore, if $A_1,A_2\subset\ZZ_{\ge0}^n$ are non-empty finite sets, we define the sum as $A_1+A_2=\{{\bf a}_1+{\bf a}_2\ :\ {\bf a}_1\in A_1,{\bf a}_2\in A_2\}$.
    
    Now, we can characterize componentwise discrete polymatroids.
    \begin{Theorem}
    	The following conditions are equivalent:
    	\begin{enumerate}
    		\item[\textup{(i)}] $\mathcal{P}$ is a componentwise discrete polymatroid.
    		\item[\textup{(ii)}] For all $\alpha(\mathcal{P})\le j\le\omega(\mathcal{P})$, the simplicial multicomplex
    		$$
    		\bigcup_{k=\alpha(\mathcal{P})}^j(\mathcal{P}_{\langle k\rangle}+[n]^{\langle j-k\rangle})
    		$$
    		is a discrete polymatroid.
    		\item[\textup{(iii)}] For all ${\bf a},{\bf b}\in\bigcup_{\ell=\alpha(\mathcal{P})}^{\omega(\mathcal{P})}\bigcup_{k=\alpha(\mathcal{P})}^\ell(\mathcal{P}_{\langle k\rangle}+[n]^{\langle \ell-k\rangle})$ with $\alpha(\mathcal{P})\le|{\bf a}|\le|{\bf b}|$ and ${\bf a}\not\le{\bf b}$, and all $i$ such that ${\bf b}[i]>{\bf a}[i]$, there is an integer $j$ with ${\bf b}[j]<{\bf a}[j]$ such that ${\bf b}-{\bf e}_i+{\bf e}_j\in\bigcup_{\ell=\alpha(\mathcal{P})}^{\omega(\mathcal{P})}\bigcup_{k=\alpha(\mathcal{P})}^\ell(\mathcal{P}_{\langle k\rangle}+[n]^{\langle \ell-k\rangle})$.
    	\end{enumerate}
    \end{Theorem}
    \begin{proof}
    	We first notice the following fact. Let $I\subset S$ be a monomial ideal, and let $\omega(I)=\max\{\deg(u):u\in\mathcal{G}(I)\}$. Then $I$ is componentwise polymatroidal if and only if $I_{\langle j\rangle}$ is polymatroidal for $\alpha(I)\le j\le\omega(I)$. Only sufficiency needs a proof. Suppose that $I_{\langle j\rangle}$ is polymatroidal for $\alpha(I)\le j\le\omega(I)$. If $j>\omega(I)$, then $I_{\langle j\rangle}=\mathfrak{m}^{j-\omega(I)}I_{\langle\omega(I)\rangle}$ is polymatroidal for it is the product of two polymatroidal ideals.
    	
    	It is easily seen that $I(\mathcal{P})_{\langle j\rangle}=I(\bigcup_{k=\alpha(\mathcal{P})}^j(\mathcal{P}_{\langle k\rangle}+[n]^{\langle j-k\rangle}))$ for all $\alpha(\mathcal{P})\le j\le\omega(\mathcal{P})$. Since, by definition, $I(\mathcal{P})$ is componentwise polymatroidal if and only if $I(\mathcal{P})_{\langle j\rangle}$ is polymatroidal for all $\alpha(\mathcal{P})\le j\le\omega(\mathcal{P})$, the equivalence (i)$\Leftrightarrow$(ii) follows at once.
    	
    	The implication (i)$\Rightarrow$(iii) follows from Theorem \ref{Thm:BQ}. Conversely, assume that (iii) holds. Then, \cite[Theorem 2.3]{JT} implies that $I(\mathcal{P})_{\langle j\rangle}$ is polymatroidal for all $\alpha(\mathcal{P})\le j\le\omega(\mathcal{P})$. This shows that (iii)$\Rightarrow$(ii) and concludes the proof.
    \end{proof}
	
	A simplicial multicomplex $\mathcal{M}$ is called \textit{pure} if $|{\bf a}|=|{\bf b}|$ for all ${\bf a},{\bf b}\in\mathcal{F}(\mathcal{M})$. Whereas, $\mathcal{M}$ is called \textit{shellable} if there exists an order ${\bf a}_1,\dots,{\bf a}_m$ of $\mathcal{F}(\mathcal{M})$ such that the simplicial multicomplex $$\langle{\bf a}_1,\dots,{\bf a}_{j-1}\rangle\cap\langle{\bf a}_j\rangle$$ is pure of dimension $|{\bf a}_j|-1$ for all $j=2,\dots,m$. In this case, ${\bf a}_1,\dots,{\bf a}_m$ is called a \textit{shelling order} of $\mathcal{M}$. It is well-known and easily seen that ${\bf a}_1,\dots,{\bf a}_m$ is a shelling order of $\mathcal{M}$ if and only if ${\bf x}^{{\bf a}_1},\dots,{\bf x}^{{\bf a}_m}$ is a linear quotients order of $I(\mathcal{M})$. Thus, Theorem \ref{Thm:CMLQ} implies immediately
	\begin{Corollary}
		Componentwise discrete polymatroids are shellable.
	\end{Corollary}
	
	We end the paper with some natural questions.
	
	Let $\mathcal{P}$ be a componentwise discrete polymatroid. Attached to $\mathcal{P}$ there are the following three monomial subalgebras of $S[t]$:
	\begin{align*}
		K[\mathcal{P}]\ &=\ K[{\bf x^a}t\ :\ {\bf a}\in\mathcal{P}],\\[3pt]
		K[\mathcal{F}(\mathcal{P})]\ &=\ K[{\bf x^a}t\ :\ {\bf a}\in\mathcal{F}(\mathcal{P})],\\
		\mathcal{R}(I(\mathcal{P}))\ &=\ \bigoplus_{k\ge0}I(\mathcal{P})^kt^k\ =\ K[x_1,\dots,x_n,{\bf x^a}t\ :\ {\bf a}\in\mathcal{F}(\mathcal{P})].
	\end{align*}
    We call $K[\mathcal{F}(\mathcal{P})]$ the \textit{base ring} of $\mathcal{P}$. Whereas, $\mathcal{R}(I(\mathcal{P}))$ is the Rees algebra of $I(\mathcal{P})$. These three algebras are toric rings. It follows from a famous theorem of Hochster that if a toric ring is normal, then it is Cohen--Macaulay \cite{Hoc72}.
	
	\begin{Question}
		\rm Let $\mathcal{P}$ be a componentwise discrete polymatroid. Are the rings $K[\mathcal{P}],K[\mathcal{F}(\mathcal{P})],\mathcal{R}(I(\mathcal{P}))$ normal? Cohen--Macaulay?
	\end{Question}

    The above question has a positive answer when $\mathcal{P}$ is actually a discrete polymatroid, see \cite[Theorem 6.1]{JT}, \cite[Corollary 6.2]{JT} and \cite[Proposition 3.11]{VRees}.
    
    On the other hand, the following question is open even for discrete polymatroids.
	
	\begin{Question}
		\rm Let $\mathcal{P}$ be a componentwise discrete polymatroid. Are the rings $K[\mathcal{P}],K[\mathcal{F}(\mathcal{P})],\mathcal{R}(I(\mathcal{P}))$ Koszul?
	\end{Question}

\end{document}